\documentclass[a4paper,dvipdfmx,12pt]{amsart}
\usepackage{amsmath}
\usepackage{amssymb}
\usepackage{amscd}
\usepackage{mathrsfs}
\usepackage{url}
\usepackage[all]{xy}
\usepackage[dvipdfmx]{graphicx}
\usepackage{graphics}
\usepackage{tikz}
\everymath{\displaystyle}
\theoremstyle{plain} 
\newtheorem{theorem}{\indent\sc Theorem}[section]
\newtheorem{lemma}[theorem]{\indent\sc Lemma}

\newtheorem{proposition}[theorem]{\indent\sc Proposition}

\theoremstyle{definition} 

\newtheorem{remark}[theorem]{\indent\sc Remark}

\newtheorem{question}[theorem]{\indent\sc Question}

\newcommand{\ve}{\varepsilon}
\newcommand{\vp}{\varphi}

\begin{document}

\pagestyle{plain}
\thispagestyle{plain}

\title[Neighborhood of a rational curve with an ordinary cusp]
{Neighborhood of a rational curve with an ordinary cusp}

\author[Takayuki Koike and Takato Uehara]{Takayuki Koike$^{1}$ and Takato Uehara$^{2}$}
\address{ 
$^{1}$ Department of Mathematics \\
Graduate School of Science \\
Osaka Metropolitan University \\
3-3-138 Sugimoto \\
Osaka 558-8585 \\
Japan 
}
\email{tkoike@omu.ac.jp}
\address{
$^{2}$ Department of Mathematics \\
Faculty of Science \\
Okayama University \\
1-1-1, Tsushimanaka \\
Okayama, 700-8530 \\
Japan
}
\email{takaue@okayama-u.ac.jp}
\renewcommand{\thefootnote}{\fnsymbol{footnote}}
\renewcommand{\thefootnote}{\arabic{footnote}}

\begin{abstract}
We investigate complex analytic properties of a neighborhood of a reduced rational curve with an ordinary cusp embedded in  a non-singular complex surface whose self-intersection number is zero. 
\end{abstract}

\maketitle

\section{Introduction}\label{section:introduction}

In this paper, we investigate complex analytic properties of a neighborhood of a compact reduced complex curve $C$ of a non-singular complex surface $S$ when it is a rational curve with an ordinary cusp. 
As it is known that $C$ admits a fundamental system of strongly pseudoconvex neighborhoods when the self-intersection number $(C^2)$ of $C$ is negative \cite{G} and of strongly pseudoconcave neighborhoods when $(C^2)$ is positive \cite{U}, 
we investigate the case of $(C^2)=0$: i.e. when the normal bundle $N_{C/S}:=j^*[C]$ is topologically trivial, where $[C]$ is the holomorphic line bundle on $S$ which corresponds to the divisor $C$ and $j\colon C\to S$ is the inclusion. 

Neighborhoods of a curve $C\subset S$ with $(C^2)=0$ is first investigated by Ueda for $C$ non-singular in \cite{U83}. 
When $C$ is non-singular, any topologically trivial holomorphic line bundle on $C$ admits a unitary flat structure (a theorem of Kashiwara, see \cite[\S 1]{U83} for example). 
By focusing on this fact, he classified the situation in accordance with the unitary-linearlizability of the transition functions of local defining functions of $C$, whose obstruction in $n$-jet is represented by {\it $n$-th Ueda class} $u_n(C, S)\in H^1(C, \mathcal{O}_C(N_{C/S}^{-n}))$ for an integer $n>0$. 
When $u_n(C, S)\not=0$ for a positive integer $n$ (i.e. the transition functions of local defining functions of $C$ cannot be linearized even formally), the pair $(C, S)$ is said to be {\it of finite type}. 
Otherwise, the pair $(C, S)$ is said to be {\it of infinite type}. 
In \cite{U83}, (strongly) plurisubharmonic functions on $V\setminus C$ are closely investigated for a neighborhood $V$ of $C$ when $(C, S)$ is of finite type, 
and some sufficient conditions for the unitary-linearlizability of the transition functions of local defining functions of $C$ is given when $(C, S)$ is of infinite type. 
He also showed the existence of $(C, S)$ of infinite type such that the transition functions of local defining functions of $C$ are {\it not} unitary-linearizable. 
See \cite{GS} and \cite{K20} for studies towards higher (co-)dimensional generalizations of Ueda's results for infinite type pairs. 

For $C$ singular, $N_{C/S}$ need not be unitary flat even under the assumption that it is topologically trivial. 
For the case where $C$ admits only nodes as singular points, neighborhoods of a rational curve with a node was first investigated by Ueda in \cite{U91}. 
In \cite{K17}, some generalizations of the results in \cite{U83} and \cite{U91} were made by the first author mainly for nodal curves when the dual graphs are either trees or cycles. 
On the other hand, the case where $C$ has singular points other than nodes does not seem to be well-understood. 

In the present paper, we investigate the case where $C$ is a rational curve with a cusp. 
More precisely, we let $S$ be a non-singular surface and $C\subset S$ be a reduced compact curve with topologically trivial normal bundle which is {\it a rational curve with an ordinary cusp}: i.e. it is isomorphic to the curve $\{[x; y; z]\in \mathbb{P}^2 \mid x^3-y^2z=0\}$, where $\mathbb{P}^2$ denotes the complex projective plane. 
Note that the group ${\rm Pic}^0(C)$ of topologically trivial holomorphic line bundles on $C$ is isomorphic to $\mathbb{C}$, and that $N_{C/S}\in {\rm Pic}^0(C)$ admits a unitary flat structure if and only if it is holomorphically trivial (see Lemma \ref{lem:Pic0}). 
When $N_{C/S}$ is holomorphically trivial, we will see that Ueda classes $u_n(C, S)\in H^1(C, \mathcal{O}_C)$ can be defined in almost the same manner as in \cite{U83}, by using which we classify the pair $(C, S)$ into {\it finite type} and {\it infinite type} (see \S \ref{section:obst} in the present paper for the precise definition). 
In accordance with this classification, we show the following: 
\begin{theorem}\label{thm:main}
Let $S$ be a non-singular surface and $C\subset S$ be a (reduced) rational curve with an ordinary cusp whose normal bundle is topologically trivial. 
Then the following holds. 
\begin{enumerate}
\item[$(i)$] If $N_{C/S}$ is holomorphically trivial and the pair $(C, S)$ is infinite type, there exists a neighborhood $V$ of $C$ in $S$ and an elliptic fibration $f\colon V \to \Delta$ onto the unit disc $\Delta\subset \mathbb{C}$ such that $C$ is a fiber of $f$. \\
\item[$(ii)$] If either $N_{C/S}$ is not holomorphically trivial or ($N_{C/S}$ is holomorphically trivial and) the pair $(C, S)$ is finite type, there exists a positive rational number $\ell\in\textstyle\frac{1}{6}\mathbb{Z}$ such that the following assertions $(1)$ and $(2)$ hold: \\
\begin{enumerate}
\item[$(1)$] For each real number $\lambda>1$, there exists a neighborhood $V$ of $C$ and a strongly plurisubharmonic function $\Phi_\lambda\colon V\setminus C\to \mathbb{R}$ such that $\Phi_\lambda(p)$ increases with the same order as $|f_C(p)|^{-\lambda \ell}$ when $p$ approaches $C$ on a neighborhood of each point of $C$, where $f_C$ is a local (holomorphic) defining function of $C$. 
In particular, $C$ has a fundamental system of strongly pseudoconcave neighborhoods. \\
\item[$(2)$] Let $V$ be a neighborhood of $C$ and $\Psi$ be a plurisubharmonic function on $V\setminus C$. If there exists a real number $0<\lambda<1$ such that $\Psi(p)=o(|f_C(p)|^{-\lambda\ell})$ holds on a neighborhood of each point of $C$, then $\Psi|_{V_0\setminus C}$ is a constant function for a neighborhood $V_0$ of $C$ in $V$. 
\end{enumerate}
\end{enumerate}
\end{theorem}

The assertion $(i)$ of Theorem \ref{thm:main} and Theorem \ref{thm:neighborhood_of_infinite_type} in the present paper, which is needed to show this theorem, can be regarded as an analogue of \cite[Theorem 3]{U83} for non-singular curves and \cite[Theorem 1.4]{K17} for nodal curves. 
The assertion $(ii)$ can be regarded as an analogue of \cite[Theorem 1, 2]{U83} at least when $N_{C/S}$ is holomorphically trivial. See \S \ref{section:examples} for further discussions on comparison with previous results. 

As an application of Theorem \ref{thm:main}, we have the following for the case where $S$ is compact K\"ahler. 
\begin{theorem}\label{thm:cpt}
Let $S$ be a connected compact K\"ahler surface and $C\subset S$ be a (reduced) rational curve with an ordinary cusp whose normal bundle is topologically trivial. 
Then either of the following $(i)$ or $(ii)$ holds: 
\begin{enumerate}
\item[$(i)$] There exists an elliptic fibration $f\colon S \to R$ onto a compact Riemann surface $R$ such that $C$ is a fiber of $f$, or\\
\item[$(ii)$] The complement $S\setminus C$ is strongly $1$-convex: i.e there exists an exhaustion function on $S\setminus C$ which is strongly plurisubharmonic outside a compact subset. 
\end{enumerate}
In particular, the complement $S\setminus C$ is holomorphically convex. 
Moreover, the following are equivalent:
\begin{enumerate}
\item[$(1)$] The assertion $(i)$ holds. \\
\item[$(2)$] The normal bundle $N_{C/S}$ is holomorphically trivial and the pair $(C, S)$ is of infinite type. \\
\item[$(3)$] The line bundle $[C]$ admits a $C^\infty$ Hermitian metric $h$ such that the Chern curvature $\sqrt{-1}\Theta_h$ is positive semi-definite at any point of $S$. 
\end{enumerate}
\end{theorem}

To prove Theorem \ref{thm:main}, we first show an Ueda-type linearization result (Theorem \ref{thm:neighborhood_of_infinite_type}) for neighborhoods of a rational curve with an ordinary cusp. 
Theorem \ref{thm:main}  is shown by combining this result and an argument in \cite[\S 7]{K17} on some concrete examples (the argument to reduce the problem to that on a neighborhood of a non-singular elliptic curve by changing the model by taking a blow-up, a suitable branched covering, and a blow-down). 

The organization of the paper is as follows. 
\S 2 is a preliminary section. 
In \S 3, we give the definition of obstruction classes and types for the pair of a non-singular surface and a rational curve with an ordinary cusp. 
In \S 4, we show a Ueda-type linearization result for neighborhoods of a rational curve with an ordinary cusp. 
Theorems \ref{thm:main} and \ref{thm:cpt} are shown in \S 5. 
In \S 6, we give an example and make further discussions. 

\vskip3mm
{\bf Acknowledgment. } 
The first author is supported by the Grant-in-Aid for Scientific Research C (23K03119) from JSPS. 
The second author is supported by the Grant-in-Aid for Scientific Research C (23K03148) from JSPS. 

\section{Preliminaries}\label{section:prelim}

\subsection{Open coverings and local coordinates}\label{section:notation}
Here we explain our choice of open coverings of (a neighborhood of) a rational curve with an ordinary cusp, and also fix notation concerning local coordinates that we will use throughout this paper. 

Unless otherwise specified, $C$ denotes the rational cuspidal curve $\{[x; y; z]\in \mathbb{P}^2 \mid x^3-y^2z=0\}$, which has $p_0:=[0; 0; 1]$ as an ordinary cusp. 
Let $U_0$ be the intersection of $C$ and a sufficiently small polydisc of $\mathbb{P}^2$ centered at $p_0$. 
Denote by $U_{\rm reg}$ the complement of a sufficiently small open neighborhood of $p_0$ in $U_0$.  
Taking a finite open covering $\{U_1, U_2, \dots, U_N\}$ of $U_{\rm reg}$, we use $\{U_j\}_{j=0}^N$ as an open Stein covering of $C$ in what follows. 

When $C$ is embedded in a non-singular complex surface $S$ as a reduced subvariety, we take an open neighborhood $V_0$ of $p_0$ in $S$ which can be regarded as a polydisc by suitable local coordinates $(x, y)$, and an open neighborhood $V_j$ of $U_j$ in $S\setminus \{p_0\}$ for each $j=1, 2, \dots, N$. 
We shrink the radii of the polydisc to $\ve_0^2$ and $2\ve_0^3$ for a sufficiently small positive number $\ve_0$, i.e. $V_0 = \{|x|<\ve_0^2\}\times\{|y|<2\ve_0^3\}$, and assume that  
\[
U_0 = V_0\cap C = \{(x, y)\in V_0\mid y^2 = x^3 \}
\]
holds. 
The choice of radii $\ve_0^2$ and $2\ve_0^3$ is used in the latter argument (see \S \ref{section:linearlization}). 
Denote by $\iota \colon \mathbb{P}^1 \to C$ the normalization, and put $\widetilde{U}_j := \iota^{-1}(U_j)$ for $0 \le j \le N$. 
Note that $i:=\iota|_{\widetilde{U}_0} : \Omega:= \widetilde{U}_0 \to U_0$ can be realized by letting 
\[
\Omega = \Delta_{\ve_0}:=\{\zeta \in \mathbb{C}\mid |\zeta|<\ve_0\}
\]
and $i(\zeta):=(\zeta^2, \zeta^3)\in V_0$. 
As is well-known (see e.g. \cite{HM}, 5B), 
\begin{equation}\label{eq:stalk_of_a_cusp}
\mathcal{O}_{C, p_0} \cong \left\{\left.\sum_{n=0}^\infty a_n\zeta^n\in\mathbb{C}\{\zeta\}\right| a_1=0\right\}
\end{equation}
holds via the pull-back by $i$. 

For $j=1, 2, \dots, N$, we take coordinates $(z_j, w_j)$ on a neighborhood of $\overline{V_j}$, by slightly shrinking $V_j$ if necessary, by which $V_j$ can be identified with $U_j\times \{|w_j|<\ve_1\}$ for a positive number $\ve_1$ and 
\[
U_j = V_j\cap C = \{(z_j, w_j)\in V_j\mid w_j=0 \}
\]
holds. On each $U_j$ with $j>0$, we use $z_j$ as a coordinate. 

\begin{remark}\label{rmk:refinement_needed}
By Siu's theorem \cite{S}, we can take a Stein neighborhood $V_{\rm reg}$ of $U_{\rm reg}$ which admits coordinates $(z_{\rm reg}, w_{\rm reg})$ such that 
\[
U_{\rm reg} = V_{\rm reg}\cap C = \{(z_{\rm reg}, w_{\rm reg})\in V_{\rm reg}\mid w_{\rm reg} = 0 \}
\]
holds. Apparently, this fact supports an idea that it is natural to take $N=1$, $U_1=U_{\rm reg}$, and $V_1=V_{\rm reg}$. 
However, in actual arguments (mainly in the proof of Theorem \ref{thm:main}, see \S \ref{subsection:suitable_covering}), we need to refine $U_{\rm reg}$ in order to closely compare some functions defined on $V_0$ and $V_j$ with $j>0$. 
\end{remark}

In what follows, the comparison of the defining functions $w_0$ of $U_0$ in $V_0$ ($w_0:=y^2-x^3$, for example) and $w_j$'s plays an important role for classifying the neighborhoods of $C$ in $S$. 
If $w_j=w_k$ holds on any $V_{jk}:=V_j\cap V_k$, it is clear that, by shrinking $V_j$'s if necessary, these functions glue up to define a holomorphic surjection $f\colon V:=\cup_{j=0}^N V_j \to \Delta_\ve$ onto a disc $\Delta_\ve:=\{\eta \in \mathbb{C}\mid |\eta|<\ve\}$ for a small positive number $\ve$. 
Then $f$ is proper and surjective due to the following lemma. 
\begin{lemma}\label{lem:conti_map_with_cpt_inv}
Let $Y$ be a topological space that is locally compact, namely, any point of $Y$ has an open neighborhood with compact closure, 
and $f : Y \to \mathbb{C}$ be a continuous open map such that $0 \in \mathbb{C}$ is contained in the image $f(Y)$ 
and its preimage $f^{-1}(0)$ is compact in $Y$. 
Then there exists an open neighborhood $V \subset Y$ of  $f^{-1}(0)$ and an open neighborhood $\Delta \subset \mathbb{C}$ of $0$ such that 
$f|_V : V \to \Delta$ is proper and surjective. 
\end{lemma}

\begin{proof}
From the assumption, there exists an open neighborhood $U$ of $f^{-1}(0)$ such that its closure $\overline{U}$ is compact and thus so is its boundary 
$\partial U=\overline{U} \setminus U$. 
Since $f$ is continuous, the image $f(\partial U)$ is compact in $\mathbb{C}$ with $0 \notin f(\partial U)$. 
Fix $\delta>0$ so that  $\delta< \min\{|z| \mid z\in f(\partial U)\}(>0)$ and 
\[
\Delta := \{z\in \mathbb{C} \mid |z|<\delta\} \subset f(U),
\]
and put $V:= U \cap f^{-1}(\Delta)$. 
Then $f|_V : V \to \Delta$ is surjective, and also proper since the inverse image $f|_V^{-1}(K)$ of any compact subset $K \subset \Delta$ 
is closed in $\overline{U}$ as $f|_V^{-1}(K) \cap \partial U=\emptyset$ and thus compact. 
\end{proof}

Note that the following holds for such a case. 
\begin{lemma}\label{lem:fibr_case_normal_bdl_triv}
Let $C$ and $S$ be as above. Assume that there exists a proper holomorphic surjection $f\colon V \to \Delta_\ve$ onto a disc $\Delta_\ve$ for a small positive number $\ve$, and that $C$ is a (set-theoretical) fiber of $f$. 
Then the normal bundle $N_{C/S}$ is holomorphically trivial. 
\end{lemma}

\begin{proof}
This follows from Lemma \ref{lem:Pic0} $(ii)$. 
\end{proof}

The notations given in this subsection are used in the following (sub)sections. 

\subsection{Topologically trivial holomorphic line bundles on $C$}

In this subsection, we will prepare the basic facts related to topologically trivial holomorphic line bundles on $C$. 

For $\widetilde{U}_j := \iota^{-1}(U_j)$ with the normalization $\iota\colon \mathbb{P}^1 \to C$, one has 
an isomorphism 
\[
\iota|_{\widetilde{U}_j} : \widetilde{U}_j \to U_j
\]
for $1 \le j \le N$. 
On the other hand, if $j=0$, then isomorphism (\ref{eq:stalk_of_a_cusp}) 
means that $\tilde{\alpha} \in \mathcal{O}_{\mathbb{P}^1}(\Omega)$ is realized as $\tilde{\alpha}=i^* \alpha$ for some $\alpha \in \mathcal{O}_C(U_0)$ 
if and only if $d \tilde{\alpha}/d \zeta(0)=0$. 
Any $1$-cocycle
\[
\{(U_{jk}, \beta_{jk})\} \in \check{Z}^1(\{U_j\}_{j=0}^N, \mathcal{O}_C)
\]
admits a $0$-cochain
\[
\{(U_j, \widetilde{\alpha}_j)\} \in \check{C}^0(\{\widetilde{U}_j\}_{j=0}^N, \mathcal{O}_{\mathbb{P}^1}) \quad \text{with}  \quad 
\delta\{(\widetilde{U}_j, \widetilde{\alpha}_j)\} = \iota^*\{(U_{jk}, \beta_{jk})\}, 
\]
which is unique up to an additive constant since 
\[
\check{H}^q(\{\widetilde{U}_j\}_{j=0}^N, \mathcal{O}_{\mathbb{P}^1}) = H^q(\mathbb{P}^1, \mathcal{O}_{\mathbb{P}^1}) = 
\left\{ \begin{array}{rl}
\mathbb{C} & (q=0) \\
0 & (q=1)
\end{array} \right. 
\]
for a Stein open covering $\{\widetilde{U}_j\}_{j=0}^N$ of $\mathbb{P}^1$, and thus induces
\[
S(\{(U_{jk}, \beta_{jk})\}) := \frac{d \widetilde{\alpha}_0}{d\zeta}(0)
\]
on $\widetilde{U}_0=\Omega$, which defines a linear map $S\colon \check{Z}^1(\{U_j\}_{j=0}^N, \mathcal{O}_C)\to \mathbb{C}$. 
Then, we have 
\[
{\rm Ker}\,S = \check{B}^1(\{U_j\}_{j=0}^N, \mathcal{O}_C). 
\]
In particular, it follows that $H^1(C, \mathcal{O}_C)\cong\mathbb{C}$. 
Moreover by applying the result in \cite[Lemma 2]{KS} (=\cite[Lemma 3]{U83}), 
we have the following lemma. 

\begin{lemma}\label{lem:H^1ofC}
There exists a positive constant $K_0>0$ such that any $1$-coboundary
\[
\{(U_{jk}, \beta_{jk})\} \in \check{B}^1(\{U_j\}_{j=0}^N, \mathcal{O}_C)
\]
admits a $0$-cochain
\[
\{(U_j, \alpha_j)\} \in \check{C}^0(\{U_j\}_{j=0}^N, \mathcal{O}_{\mathbb{P}^1}) \quad \text{with} \quad 
\delta\{(U_j, \alpha_j)\} = \{(U_{jk}, \beta_{jk})\}
\]
that satisfies 
\[
\max_{0\leq j\leq N}\sup_{U_j}|\alpha_j| \leq K_0\cdot \max_{0\leq j, k\leq N}\sup_{U_j\cap U_k} |\beta_{jk}|. 
\]
\end{lemma}

Note that  the exponential sheaf homomorphism $\exp : \mathcal{O}_C \to \mathcal{O}_C^*$ gives the (exponential) 
exact sequence $0\to 2\pi\sqrt{-1}\mathbb{Z}\to \mathcal{O}_C\to \mathcal{O}_C^*\to 0$, which induces the long exact sequence
\[
H^1(C,2\pi\sqrt{-1}\mathbb{Z}) \longrightarrow H^1(C,\mathcal{O}_C) \longrightarrow  H^1(C,\mathcal{O}_C^*) \overset{-2 \pi\sqrt{-1}c_1}{\longrightarrow } H^2(C,2\pi\sqrt{-1}\mathbb{Z}),
\]
where $\mathcal{O}_C^*$ is the sheaf of non-vanishing holomorphic functions on $C$, and $c_1$ stands for the first Chern class. 
Then ${\rm Pic}^0(C):= {\rm Ker} (c_1)$ is the subgroup of ${\rm Pic}(C):=H^1(C,\mathcal{O}_C^*)$ consisting of (isomorphism
classes of) topologically trivial line bundles. 
Moreover, the inclusion map ${\rm U}(1) \to \mathcal{O}_C^*$ induces a homomorphism $H^1(C, {\rm U}(1))\to {\rm Pic}(C)=H^1(C, \mathcal{O}_C^*)$, 
each element $L \in {\rm Pic}(C)$ of whose image is said to have a {\it unitary flat structure}. 

\begin{lemma}\label{lem:Pic0}
As to the group ${\rm Pic}^0(C)$, the following assertions hold. 
\begin{enumerate}
\item[$(i)$] ${\rm Pic}^0(C)\cong \mathbb{C}$. \\
\item[$(ii)$] If $L\in {\rm Pic}^0(C)$ satisfies $L^n:=L^{\otimes n}=1$ (i.e. $L^n$ is holomorphically trivial) for some $n \ge 1$, 
then $L$ is itself holomorphically trivial. \\
\item[$(iii)$] If $L\in {\rm Pic}^0(C)$ has a unitary flat structure, then $L$ is holomorphically trivial. 
\end{enumerate}
\end{lemma}

\begin{proof}
$(i)$ follows from the above long exact sequence and the facts that $H^1(C,2\pi\sqrt{-1}\mathbb{Z})=0$ and $H^1(C,\mathcal{O}_C) \cong \mathbb{C}$. 
$(ii)$ follows since ${\rm Pic}^0(C)\cong \mathbb{C}$ is torsion-free. 
$(iii)$ follows from
\[
H^1(C, {\rm U}(1))\cong {\rm Hom}_{\mathbb{Z}}(H_1(C, \mathbb{Z}),\ {\rm U}(1))
\]
by the universal coefficient theorem and the fact that $H_1(C, \mathbb{Z})=0$. 
\end{proof}

\subsection{Extension of holomorphic functions from $U_0$ to $V_0$}

In this subsection, we will prove the following lemma. 

\begin{lemma}\label{lem:Linf_estim_of_ext_from_U0_to_V0}
For any holomorphic function $f\in\mathcal{O}_C(U_0)\cap L^\infty(U_0)$ on $U_0$, there exists a holomorphic function $F\in \mathcal{O}_S(V_0)\cap L^\infty(V_0)$ on $V_0$ 
such that $F|_{U_0}=f$ and 
\begin{equation}\label{ineq:estim_of_ext_F}
\sup\{|F(x, y)| \mid (x, y)\in V_0\} \leq 3\cdot \sup\{|f(p)|\mid p\in U_0\}. 
\end{equation}
\end{lemma}

\begin{proof} For any holomorphic function $f$ on $U_0$, isomorphism (\ref{eq:stalk_of_a_cusp}) shows that 
the Taylor expansion of $i^*f:=f\circ i\colon \Omega\to \mathbb{C}$ at $0$ is given by 
\[
i^*f(\zeta) = a_0 + \sum_{n=2}^\infty a_n\zeta^n
\]
with $a_0, a_2, a_3, \dots \in\mathbb{C}$. Then we define holomorphic functions $g, h$ on $\Omega$ by 
\[
g(\zeta) := \frac{i^*f(\zeta) + i^*f(-\zeta)}{2}, \quad h(\zeta) := \frac{i^*f(\zeta) - i^*f(-\zeta)}{2 \zeta^3},
\]
which are even functions and satisfy $i^*f=g+ \zeta^3 \cdot h$. 
Thus a function $F\colon V_0\to \mathbb{C}$ given by
\[
F(x, y) := g(\sqrt{x}) + y\cdot h(\sqrt{x})=a_0 + \sum_{m=1}^\infty a_{2m}x^m + \sum_{m=1}^\infty a_{2m+1}x^{m-1}y
\]
is well-defined and holomorphic on $V_0$. 
Since $i^*x = \zeta^2, i^*y = \zeta^3$, one has $i^*F=i^*f$, and $F|_{U_0}=f$. 

Now we will give an estimate of the $L^\infty$-norm of $F$. Put $M:=\sup\{|f(p)|\mid p\in U_0\}$. 
From the definition of $g,h$, one has 
\[ 
\begin{array}{l}
|g(\zeta)|\leq \frac{1}{2}\left(|i^*f(\zeta)| + |i^*f(-\zeta)| \right) \leq M, \\[2mm]
|\zeta^3 \cdot h(\zeta) |\leq \frac{1}{2}\left(|i^*f(\zeta)| + |-i^*f(-\zeta)| \right) \leq M
\end{array}
\]
for each $\zeta\in \Omega$. 
The maximum principle says that 
\[
\sup\{|h(\zeta)| \mid |\zeta|<r\}
=\max\left\{\left.\frac{|\zeta^3 \cdot h(\zeta)|}{|\zeta|^3}\right| |\zeta|=r\right\}
\leq \frac{M}{r^3}
\]
for any $r\in (0, \ve_0)$. 
By letting $r\nearrow \ve_0$, we have 
\[
\sup\{|h(\zeta)| \mid \zeta\in \Omega\} \leq \frac{M}{\ve_0^3}, 
\]
and thus
\[
|F(x, y)| \leq |g(\sqrt{x})| + |y|\cdot |h(\sqrt{x})|
\leq M + 2\ve_0^3\cdot \frac{M}{\ve_0^3}
= 3M
\]
for any $(x, y) \in V_0 = \{|x|<\ve_0^2\}\times\{|y|<2\ve_0^3\}$. 
The lemma is established. 
\end{proof}
\section{Obstruction classes}\label{section:obst}

In this section, we use notations given in \S \ref{section:notation}. 
Namely, 
\begin{itemize}
\item the rational curve $C$ with an ordinary cusp is embedded in a non-singular complex surface $S$ as a reduced subvariety, 
\item $\{U_k\}_{k=0}^N$ is an open Stein covering of $C$ with $p_0 \in U_0$ and $p_0 \notin U_k$ for $1 \le k \le N$, 
where $p_0$ is the cusp of $C$, 
\item $w_0$ is a defining function of $U_0$ in an open Stein neighborhood $V_0$, and 
\item for $1 \le k \le N$, an open Stein neighborhood $V_k$ of $U_k$ in $S$ is chosen so that $p_0 \notin V_k$ and 
\[
U_k=V_k \cap C=\{(z_k,w_k) \in V_k \mid w_k=0\}, 
\]
where $(z_k,w_k)$ are local coordinates on $V_k$. 
\end{itemize}
Moreover, we assume that $N_{C/S} \in H^1(C, \mathcal{O}_C^*)$ is holomorphically trivial, 
and give a similar discussion given in \cite[\S 2]{U83}. 
As $N_{C/S}=[\{ (U_{jk},dw_j/dw_k|_{U_{jk}}) \}]=1$ with $U_{jk}:=U_j \cap U_k$, one may assume that $dw_j=dw_k$ on $U_{jk}$
by multiplying each defining function $w_k$ by a non-where vanishing holomorphic function. 
Then for $j>0$ and $k \geq 0$, the defining function $w_k|_{V_{jk}}$ on $V_{jk}$ can be expressed in terms of the coordinates $(z_j,w_j)$ as 
\[
w_k = w_j + \sum_{\nu=2}^\infty f_{kj, \nu}(z_j)\cdot w_j^\nu, 
\]
where $f_{kj, \nu}$ is a holomorphic function on $U_{jk}$. Moreover we put $f_{j0, \nu}:=-f_{0j, \nu}$ for $j>0$. 

Let $n \geq 1$ be a positive integer. A system $\{(V_j, w_j)\}_j$ is said to be {\it of type $n$} if $f_{kj, \nu}\equiv 0$ for any $2\leq \nu\leq n$. 
Note that any system is of type $1$. 
Then the {\it $n$-th obstruction class} (or {\it Ueda class}) $u_n(C, S)\in H^1(C, \mathcal{O}_C)$ (see \cite[\S 2]{U83}, \cite[\S 1]{N}) is given by
\[
u_n(C, S) := [\{(U_{jk}, f_{kj, n+1})\}] \in \check{H}^1(\{U_j\}_{j=0}^N, \mathcal{O}_C),
\]
which is well-defined due to the following lemmas. 

\begin{lemma}\label{lem:1-cocycle}
Any system $\{(U_{jk}, f_{kj, n+1})\} \in \check{C}^1(\{U_j\}_{j=0}^N, \mathcal{O}_C)$ satisfies 
$\delta \{(U_{jk}, f_{kj, n+1})\}=0$. 
\end{lemma}

\begin{proof}
Let $\{(U_{jk}, f_{kj, n+1})\}$ be a system of type $n$. Then, it follows that 
\[
w_k = w_j + f_{kj, n+1}(z_j)\cdot w_j^{n+1} + O(w_j^{n+2})
\]
on $V_{jk}$ for $j>0$ and $k \geq 0$, and thus 
\[
\frac{1}{w_k^n} = \frac{1}{w_j^n}\cdot \left(1 + f_{kj, n+1}(z_j)\cdot w_j^n + O(w_j^{n+1})\right)^{-n}
= \frac{1}{w_j^n} - n \cdot f_{kj, n+1}(z_j)+ O(w_j).
\]
Therefore we have
\[
f_{kj, n+1} = \frac{1}{n}\left.\left(\frac{1}{w_j^n}-\frac{1}{w_k^n}\right)\right|_{U_{jk}}, 
\]
which proves the lemma. 
\end{proof}

\begin{lemma}\label{lem:well_def_obstr_class}
Let $n \geq 1$ be a positive integer. 
Then, the choice of a system $\{(V_j, w_j)\}$ of type $n$ does not affect whether 
the $n$-th obstruction class $u_n(C, S)$ is $0$ or not. 
\end{lemma}

\begin{proof}
It is enough to compare the obstruction classes given by two systems $\{(V_j, w_j)\}$, $\{(V_j, v_j)\}$ of type $n$, 
by taking a common refinement of two open coverings. 
For $j>0$ and $k \geq 0$, we put 
\[
\begin{array}{ll}
w_k & = w_j + f_{kj, n+1}(z_j)\cdot w_j^{n+1} + O(w_j^{n+2}), \\[2mm]
v_k & = v_j + g_{kj, n+1}(z_j)\cdot v_j^{n+1} + O(v_j^{n+2}),
\end{array}
\]
where $f_{kj, n+1}$ and $g_{kj, n+1}$ are holomorphic functions on $U_{jk}$. 
Since $(v_k/w_k)|_{U_k}$ is a holomorphic function without zeros and the above two equations yield 
\[
\frac{v_k}{w_k}=\frac{v_j}{w_j} \left(1+O(v_j,w_j) \right), 
\]
the system $\{(v_k/w_k)|_{U_k}\}$ defines a global holomorphic function on $C$, 
which is constant as $C$ is compact. 
So one may assume that $(v_k/w_k)|_{U_k} \equiv 1$ by multiplying each $v_k$ by a common nonzero constant, 
by which whether the obstruction class is $0$ or not remains unchanged. 
Then we have $v_k/w_k=1+h_{k,2} \cdot w_k$, or 
\begin{equation} \label{eq:relationvw2}
v_k=w_k+h_{k,2} \cdot w_k^2
\end{equation}
for any $k=0,1,\dots,N$, where $h_{k,2}$ is a holomorphic function on $V_k$. 
In particular, it follows that
\begin{equation} \label{eq:expressionvw}
v_k-v_j=g_{kj, n+1}(z)\cdot v_j^{n+1} + O(v_j^{n+2}) = g_{kj, n+1}(z)\cdot w_j^{n+1} + O(w_j^{n+2}). 
\end{equation}

Now, for a given integer $1 \leq \mu \leq n$,  we assume that 
\[
\mathrm{(A)_{\mu}} \hspace{10mm} v_k = w_k + c_2\cdot w_k^2 + \cdots + c_\mu\cdot w_k^\mu + h_{k,\mu+1}\cdot w_k^{\mu+1}
\]
for any $k=0,1,\dots,N$, where $c_2,\dots, c_{\mu} \in \mathbb{C}$ are constants independent of $k$, 
and $h_{k,\mu+1}$ is a holomorphic function on $V_k$ (relation (\ref{eq:relationvw2}) means that assertion $\mathrm{(A)_{1}}$ holds). 
Then one has
\[
\begin{array}{rl}
\displaystyle v_k = & w_k+\sum_{\nu=2}^{\mu} c_{\nu} \cdot w_k^\nu + h_{k,\mu+1}\cdot w_k^{\mu+1} \\[2mm]
= & \left\{ w_j + f_{kj, n+1}(z)\cdot w_j^{n+1} + O(w_j^{n+2}) \right\}+ \sum_{\nu=2}^{\mu} c_{\nu} \cdot \left\{ w_j + f_{kj, n+1}(z)\cdot w_j^{n+1} + O(w_j^{n+2}) \right\}^{\nu} \\
~& \hspace{65mm} + h_{k,\mu+1}\cdot \left\{ w_j + f_{kj, n+1}(z)\cdot w_j^{n+1} + O(w_j^{n+2}) \right\}^{\mu+1} \\[2mm]
= & w_j+\sum_{\nu=2}^{\mu} c_{\nu} \cdot w_j^\nu+ 
\left\{
\begin{array}{ll}
h_{k,\mu+1} \cdot w_j^{\mu+1} +O(w_j^{\mu+2}) & \text{if~} \mu<n \\[2mm]
(h_{k,n+1}+f_{kj,n+1}) \cdot w_j^{n+1} +O(w_j^{n+2}) & \text{if~} \mu=n
\end{array}
\right.
\end{array}
\]
and thus
\[
v_k-v_j=
\left\{
\begin{array}{ll}
(h_{k,\mu+1}-h_{j,\mu+1}) \cdot w_j^{\mu+1} +O(w_j^{\mu+2}) & \text{if~} \mu<n \\[2mm]
(h_{k,n+1}+f_{kj,n+1}-h_{j,n+1}) \cdot w_j^{n+1} +O(w_j^{n+2}) & \text{if~} \mu=n. 
\end{array}
\right.
\]
By comparing the above relation to (\ref{eq:expressionvw}), we have the following. 
If $\mu<n$, then $\{h_{k,\mu+1}|_{U_k}\}_k$ defines a global holomorphic function on $C$, 
which is constant: $h_{k,\mu+1}|_{U_k} \equiv c_{\mu+1}\in\mathbb{C}$, and thus satisfies $h_{k,\mu+1}=c_{\mu+1}+h_{k,\mu+2} \cdot w_k$, 
which means that $\mathrm{(A)_{\mu+1}}$ holds. 
On the other hand, if $\mu=n$, then we have
\[
\{(U_{jk}, g_{kj, n+1})\}-\{(U_{jk}, f_{kj, n+1})\} = \delta\{(U_j, h_{j,n+1}|_{U_j})\},
\]
which proves the lemma.  
\end{proof}


\begin{lemma}\label{lem:well-def_type}
Let $n \geq 1$ be a positive integer, and assume that there exists a system of type $n$. 
Then there exists a system of type $n+1$ if and only if $u_n(C, S)=0$. 
\end{lemma}

\begin{proof}
It is trivial that if there exists a system of type $n+1$ then $u_n(C, S)=0$. 
So assume that $u_n(C, S)=0$. 
Let $\{(V_j, w_j)\}$ be a system of type $n$ with
\[
w_k = w_j + f_{kj, n+1}(z_j)\cdot w_j^{n+1} + O(w_j^{n+1})
\]
on $V_{jk}$ for $j>0$ and $k \geq 0$. 
From the assumption, there exists a holomorphic function $g_k$ on $U_k$ such that $-g_j+g_k=f_{kj, n+1}$ holds 
on $U_{jk}$ for any $j,k \geq 0$. 
Since $V_k$ is Stein, there exists a holomorphic function $G_k$ on $V_k$ such that $g_k=G_k|_{U_k}$. 
By putting
$v_k := w_k - G_k\cdot w_k^{n+1}$, 
we have 
\begin{align*}
v_k &= w_k - G_k\cdot w_k^{n+1}\\
&= w_j + f_{kj, n+1}(z)\cdot w_j^{n+1}- g_k \cdot w_j^{n+1} +O(w_j^{n+2})\\
&= w_j - g_j \cdot w_j^{n+1} +O(w_j^{n+2})\\
&= w_j - G_j \cdot w_j^{n+1} +O(w_j^{n+2})=v_j + O(v_j^{n+2})
\end{align*}
on $V_{jk}$  for $j>0$ and $k \geq 0$. The lemma is proved. 
\end{proof}

The above lemma shows that the following dichotomy holds:
\begin{itemize}
\item There exists a positive integer $n$, depending only on $(C, S)$ from Lemma \ref{lem:well_def_obstr_class}, 
such that $(C, S)$ admits a system of type $n$ but $u_n(C, S)\not=0$. \\
\item For any positive integer $n$, $(C, S)$ admits a system of type $n$ and $u_n(C, S)=0$. 
\end{itemize}
In the former case, $(C, S)$ is said to be {\it of finite type}, while in the latter case, it is said to be {\it of infinite type}. 

\section{Linearization of a neighborhood}\label{section:linearlization}
In this section, we will prove the following theorem. 

\begin{theorem}\label{thm:neighborhood_of_infinite_type}
Let $C$ be a rational curve with an ordinary cusp embedded in a non-singular complex surface $S$, and assume that $N_{C/S}$ is holomorphically trivial and $(C,S)$ is of infinite type. 
Then there exist an open covering $\{W_j\}$ of a neighborhood of $C$ and a defining function $u_j$ of $C\cap W_j$ on each $W_j$ such that $u_k=u_j$ on $W_{jk}=W_j\cap W_k$ for any $j$ and $k$. 
\end{theorem}

This theorem can be shown by the similar argument in \cite[\S 4]{U83}, although there is a difficulty coming from the fact that $w_0$ is not a coordinate function. 
This difficulty can be avoided by basically the same argument as in \cite[\S 4.2]{K17} by using carefully constructed local coordinates on neighborhoods of the ordinary cusp, which will be explained in \S \ref{subsection:suitable_covering}. 

\subsection{A suitable covering and coordinates}\label{subsection:suitable_covering}

Let us use the notations $U_0, V_0, U_{\rm reg}$, $V_{\rm reg}$ and $(x, y)$, $(z_{\rm reg}, w_{\rm reg})$ given in \S \ref{section:notation}, 
and consider the holomorphic function $w_{\rm sing}:=y^2-x^3$ on $V_0$. 
Under the assumption that $N_{C/S}$ is holomorphically trivial, there exist non-vanishing holomorphic functions 
$e_{\rm reg} : U_{\rm reg} \to \mathbb{C}$ and $e_{\rm sing} : U_0 \to \mathbb{C}$ such that 
$w_0 := e_{\rm sing}\cdot w_{\rm sing}$ and $w_1 := e_{\rm reg}\cdot w_{\rm reg}$ satisfy 
$dw_0=dw_1$ on  $U_{\rm reg} \cap U_0$, which defines a global section of $N_{C/S}$ on $C$. 


First fix finitely many (sufficiently small) simply connected domains $U_2, U_3, \dots, U_N$ and relatively compact ones 
$U_j^* \Subset U_j$ for $j=2,\dots,N$ so that $\{U_2^*, U_3^*, \dots, U_N^*\}$ covers the boundary $\partial U_0$ of $U_0$, which is homeomorphic to a circle. 
For each $j=2,\dots,N$, by using coordinates 
\begin{equation} \label{eqn:choiceofwj}
z_j := \sqrt{x},\quad w_j := w_0\quad (j=2, 3, \dots, N)
\end{equation}
around $U_j$, where a branch of the square roots is fixed for each $j$, we define $V_j$ to be 
the connected component of $\{(z_j, w_j)\mid (z_j, 0)\in U_j\ |w_j|<\ve_1\}$ containing $U_j$ 
with $\ve_1\ll \ve_0$. 
Moreover, simply connected domains $U_1^* \Subset U_1$ in $C$ and an open neighborhood $V_1$ of $U_1$ in $S$ are chosen so that 
$U_1$ is relatively compact in $C\setminus \overline{U_0}$, 
$\{U_0^*, U_1^*, U_2^*, \dots, U_N^*\}$ is an open Stein covering of $C$ with $U_0^*:=U_0$, $V_0\cap V_1=\emptyset$ and 
\[
V_1=\{(z_1, w_1) \mid (z_1, 0)\in U_1,\ |w_1|<\ve_1\}
\]
for coordinates $(z_1,w_1)=(z_{\rm reg}, e_{\rm reg} \cdot w_{\rm reg})$. 
Finally, fix $R>0$ so that 
\begin{equation} \label{eqn:polydiskcontV}
\left\{(z_j, w_j) \in V_j \left|\, (z_j, 0)\in U_k^*\cap U_j,\ |w_j| = \frac{1}{R}\right.\right\} \subset V_k
\end{equation}
for each $(j, k)\in \{1, 2, 3, \dots, N\}^2$ and $(j, k)\in \{2, 3, \dots, N\} \times \{0\}$. 
Note that the choice of radii $\ve_0^2$ and $2\ve_0^3$ in $V_0$ is used here.  


\begin{lemma}\label{lem:estim_KS+ext}
There exists a positive constant $K\geq1$ such that for any $1$-coboundary
\[
\{(U_{jk}, h_{jk})\} \in \check{B}^1(\{U_j\}_{j=0}^N, \mathcal{O}_C),
\]
there exist holomorphic functions $F_j \in \mathcal{O}_S(V_j)$ with $j=0, 1, \dots, N$, 
depending only on $z_j$ for $j>0$, such that 
\[
\delta\{(U_j,\ F_j|_{U_j})\} = \{(U_{jk}, h_{jk})\}
\]
and
\[
\max_{0\leq j\leq N}\sup_{V_j}|F_j| \leq K\cdot \max_{0\leq j, k\leq N}\sup_{U_j\cap U_k} |h_{jk}|. 
\]
\end{lemma}

\begin{proof}
Lemma \ref{lem:H^1ofC} says that there exists a positive constant $K_0>0$ such that any $1$-coboundary
\[
\{(U_{jk}, h_{jk})\} \in \check{B}^1(\{U_j\}_{j=0}^N, \mathcal{O}_C)
\]
admits holomorphic functions $g_j \in \mathcal{O}_C(U_j)$ ($j=0, 1, \dots, N$) that satisfy
\[
\delta\{(U_j,\ g_j)\} = \{(U_{jk}, h_{jk})\}
\qquad \text{and} \qquad 
\max_{0\leq j\leq N}\sup_{U_j}|g_j| \leq K_0\cdot \max_{0\leq j, k\leq N}\sup_{U_j\cap U_k}|h_{jk}|. 
\]
We extend $g_j$ to a holomorphic function $F_j$ on $V_j$, by using Lemma \ref{lem:Linf_estim_of_ext_from_U0_to_V0} for $j=0$, 
and by putting $F_j(z_j, w_j):=g_j(z_j)$ for $j>0$. 
Then the lemma follows for $K := \max\{ 3K_0,1\}$. 
\end{proof}

\subsection{Extensions of holomorphic functions with estimates}\label{subsection:extensions_hol_functions_estimates}

Under the assumptions in Theorem \ref{thm:neighborhood_of_infinite_type} and the notations given in \S \ref{subsection:suitable_covering}, 
the system $\{(V_j, w_j)\}_j$ satisfies  
\[
w_k = w_j + \sum_{\nu=2}^\infty f_{kj, \nu}(z_j)\cdot w_j^\nu, 
\]
for $j>0$ and $k \geq 0$, where $f_{kj, \nu}$ is a holomorphic function on $U_{jk}$.  
In this subsection, we will prove the following proposition: 
\begin{proposition}\label{prop:constr_of_Fjnu}
There exists a sequence $\{A_\nu\}_{\nu\geq 2}$ of positive real numbers and holomorphic functions $F_{j, \nu}\colon V_j\to \mathbb{C}$ 
for $j=0, 1,\dots,N$ and $\nu=2, 3, \dots$, which depend only on $z_j$ for $j \geq 1$, such that the following hold: 
\begin{enumerate}
\item[(i)] The power series $X + \textstyle\sum_{\nu=2}^\infty A_\nu X^\nu\in \mathbb{R}[[X]]$ has positive radius of convergence. \\
\item[(ii)] $\sup\{|F_{j, \nu}(p)| \mid p\in V_j\} \leq A_\nu$ holds for any $j=0, 1,\dots,N$ and $\nu=2, 3, \dots$. \\
\item[(iii)] For any $n \geq 2$, let $\{u_j\}_{j \ge 0}$ be holomorphic functions in neighborhoods of $U_j$'s obtained as solutions of the functional equations
\begin{equation}\label{eq:func_eq}
w_j = u_j + \sum_{\nu=2}^{n} F_{j, \nu}\cdot u_j^\nu \qquad (j=0,1,\dots,N)
\end{equation}
on $V_j$. 
Then $\{u_j\}_{j \ge 0}$ satisfy $u_k=u_j+O(u_j^{n+1})$ for any $j,k \geq 0$. 
\end{enumerate}
\end{proposition}

\begin{proof}
First we notice that from the choice of coordinates given in (\ref{eqn:choiceofwj}),
that of $R$ given in (\ref{eqn:polydiskcontV}) and the Cauchy estimates,  there exists a positive number 
$M>0$ such that 
\begin{equation}\label{ineq:estim_of_fnu}
\begin{cases}
\sup\{|f_{kj, \nu}(z_j)| \mid (z_j, 0)\in U_j\cap U_k^* \} \leq MR^\nu & \text{if}\ k>0\\
f_{kj, \nu} \equiv 0 & \text{if}\ k=0 
\end{cases}
\end{equation}
holds for any $\nu \geq 2$. 
With the constant $K\geq1$ given in Lemma \ref{lem:estim_KS+ext}, 
choose a sequence $\{A_\nu\}_{\nu=2}^\infty$ of positive numbers so that the series
\[
A(X) := X + \sum_{\nu=2}^\infty A_\nu\cdot X^\nu
\]
is a solution of the functional equation
\begin{equation}\label{eq:func_eq_def_A}
A(X) - X = 2KR(1+MR)\cdot \frac{(A(X))^2}{1-RA(X)}. 
\end{equation} 
From inductive arguments, it is easily seen that functional equation (\ref{eq:func_eq_def_A}) uniquely determines $A_\nu$,
which is a positive number (e.g. $A_2=2KR(1+MR)$). 
Moreover, it follows from the implicit function theorem that $A(X)$ has positive radius of convergence. 
This proves assertion (i). 

Lemma \ref{lem:estim_KS+ext} shows that there exist holomorphic functions $F_{j, 2} \in \mathcal{O}_S(V_j)$ with $j=0,1,\dots,N$ 
that satisfy $-F_{j, 2}+F_{k, 2}=f_{kj, 2}$ on $U_{jk}$ and 
\[
\max_{0\leq j\leq N}\sup_{V_j}|F_{j, 2}| \leq K\cdot \max_{0\leq j, k\leq N}\sup_{U_j\cap U_k} |f_{kj, 2}|. 
\]
Since $\{U_0^*, U_1^*, \dots, U_N^*\}$ is an open covering of $C$, 
for each point $p\in U_j\cap U_k$, there exists $\ell\in \{0, 1, 2, \dots, N\}$ such that $p\in U_\ell^*$. 
Then estimate (\ref{ineq:estim_of_fnu}) shows that 
\[
|f_{kj, 2}(p)|
=|-f_{\ell k, 2}(p) + f_{\ell j, 2}(p)| \leq \sup_{U_k\cap U_\ell^*}|f_{\ell k, 2}|
+ \sup_{U_j\cap U_\ell^*}|f_{\ell j, 2}|
\leq 2MR^2,
\]
which proves assertion (ii) for $\nu=2$. 
Moreover it follows from the implicit function theorem that the functional equation (\ref{eq:func_eq}) with $n=2$ has solutions $\{u_j\}_{j \ge 0}$, 
which satisfy $u_k=u_j+O(u_j^3)$ for any $j,k \geq 0$. 
This proves assertion (iii) for $n=2$. 

Now, for a positive integer $n \geq 2$, assume that assertions (ii) for $\nu=2,\dots,n$ and (iii) for $n$ hold. 
The implicit function theorem says that there exist holomorphic functions $\{v_j\}_{j \ge 0}$ in neighborhoods of $U_j$'s 
obtained as solutions of the functional equations
\[
w_j = v_j + \sum_{\mu=2}^{n} F_{j, \mu}\cdot v_j^\mu \qquad (j=0,1,\dots,N)
\]
on $V_j$. Then $\{v_j\}_{j \ge 0}$ satisfy 
\begin{equation}\label{eq:u_exp_equalupton+1}
v_k=v_j+O(v_j^{n+1})
\end{equation}
on $U_{jk}$ for any $j,k \geq 0$ from assertion (iii). 
By expanding $F_{k, \nu}$ on $V_{jk}$ in terms of coordinates $(z_j,w_j)$ as 
\[
F_{k, \nu} = \sum_{m=0}^\infty G_{kj, \nu, m}(z_j)\cdot w_j^m
\]
for $j \geq 1$ and $k \geq 0$, we have 
\begin{align*}
w_k &= v_k + \sum_{\nu=2}^n F_{k, \nu}\cdot v_k^\nu 
= v_k + \sum_{\nu=2}^n F_{k, \nu}\cdot v_j^\nu + O(v_j^{n+2}) \\
&= v_k + \sum_{\nu=2}^n \sum_{m=0}^\infty G_{kj, \nu, m}(z_j)\cdot w_j^m\cdot v_j^\nu + O(v_j^{n+2}) \\
&= v_k + \sum_{\nu=2}^n \sum_{m=0}^\infty G_{kj, \nu, m}(z_j)\cdot \left(v_j + \sum_{\mu=2}^n F_{j, \mu}(z_j)\cdot v_j^\mu\right)^m\cdot v_j^\nu + O(v_j^{n+2}) 
\end{align*}
and
\begin{align*}
w_k &= w_j + \sum_{\nu=2}^\infty f_{kj, \nu}(z_j)\cdot w_j^\nu \\
&= v_j + \sum_{\nu=2}^n F_{j, \nu}(z_j)\cdot v_j^\nu + \sum_{\nu=2}^\infty f_{kj, \nu}(z_j)\cdot \left(v_j + \sum_{\mu=2}^n F_{j, \mu}(z_j)\cdot v_j^\mu\right)^\nu, 
\end{align*}
which leads to 
\begin{align}\label{eq:main}
-v_k+v_j \hspace{-4mm} & \\
= &\sum_{\nu=2}^n \sum_{m=0}^\infty G_{kj, \nu, m}(z_j)\cdot \left(v_j + \sum_{\mu=2}^n F_{j, \mu}(z_j)\cdot v_j^\mu\right)^m\cdot v_j^\nu \nonumber \\
&-\sum_{\nu=2}^n F_{j, \nu}(z_j)\cdot v_j^\nu - \sum_{\nu=2}^\infty f_{kj, \nu}(z_j)\cdot \left(v_j + \sum_{\mu=2}^n F_{j, \mu}(z_j)\cdot v_j^\mu\right)^\nu 
+ O(v_j^{n+2}) \nonumber \\
= &\sum_{\nu=2}^n F_{k, \nu}(z_j)\cdot v_j^\nu 
+\sum_{\nu=2}^n \sum_{m=1}^\infty G_{kj, \nu, m}(z_j)\cdot \left(v_j + \sum_{\mu=2}^n F_{j, \mu}(z_j)\cdot (v_j)^\mu\right)^m\cdot v_j^\nu \nonumber \\
&-\sum_{\nu=2}^n F_{j, \nu}(z_j)\cdot v_j^\nu - \sum_{\nu=2}^\infty f_{kj, \nu}(z_j)\cdot \left(v_j + \sum_{\mu=2}^n F_{j, \mu}(z_j)\cdot v_j^\mu\right)^\nu 
+ O(v_j^{n+2}), \nonumber
\end{align}
since $G_{kj, \nu, 0}(z_j) = F_{k, \nu}(z_j)$. Put
\[
H_{kj, \ell}(z_j) := 
\left[\!\!\left[\sum_{\nu=2}^n \sum_{m=1}^\infty G_{kj, \nu, m}(z_j)\cdot \left(X + \sum_{\mu=2}^n F_{j, \mu}(z_j)\cdot X^\mu\right)^m\cdot X^\nu\  \right]\!\!\right]_\ell, 
\]
\[
I_{kj, \ell}(z_j) := 
\left[\!\!\left[\sum_{\nu=2}^\infty f_{kj, \nu}(z_j)\cdot \left(X + \sum_{\mu=2}^n F_{j, \mu}(z_j)\cdot X^\mu\right)^\nu \, \right]\!\!\right]_\ell, 
\]
and 
\[
J_{kj, \ell} := - H_{kj, \ell} + I_{kj, \ell}, 
\]
where for $B(X) \in \mathbb{C}[[X]]$, $[\![B(X)]\!]_\ell$ stands for the coefficient of $X^\ell$ in $B(X)$. 
Note that $H_{kj, \ell}, I_{kj, \ell}$ and $J_{kj, \ell}$ are holomorphic functions on $U_{jk}$ depending only on $\{F_{j, \nu}\}_{\nu<\ell}$ 
for $\ell=2, 3, \dots, n+1$. 
Relations (\ref{eq:u_exp_equalupton+1}) and (\ref{eq:main}) mean that 
\begin{equation}\label{eq:J_ueda_class_1cocycle}
v_k = v_j +J_{kj, n+1}(z_j)\cdot v_j^{n+1} + O(v_j^{n+2}), 
\end{equation}
and assumption $u_n(C, S)=0$ shows that 
\[
\{(U_{jk}, J_{kj, n+1})\} \in\check{B}^1(\{U_j\}_{j=0}^N, \mathcal{O}_C). 
\]
It follows from Lemma \ref{lem:estim_KS+ext} and the following Lemma \ref{lem:estim_IandH} that 
there exist holomorphic functions $F_{n+1, j} \in \mathcal{O}_S(V_j)$ for $j=0,1,\dots,N$, which depend only on $z_j$ for $j \geq 1$, such that 
\[
\delta\{(U_j,\ F_{n+1, j}|_{U_j})\} = \{(U_{jk}, J_{kj, n+1})\}
\]
and
\[
\max_{0\leq j\leq N}\sup_{V_j}|F_{n+1, j}| \leq A_{n+1},
\]
which proves assertion (ii) for $\nu=n+1$. Moreover, it follows from the discussion inducing relation (\ref{eq:J_ueda_class_1cocycle}) 
that assertion (iii) for $n+1$ holds. 
\end{proof}

\begin{lemma}\label{lem:estim_IandH}
We have the estimate
\[
\max_{0\leq j, k\leq N}\sup_{U_{jk}}|J_{kj, n+1}|
\leq \frac{A_{n+1}}{K}. 
\]
\end{lemma}

\begin{proof}
First note that it follows from the choice of $R$ given in (\ref{eqn:polydiskcontV}) and the Cauchy estimates that  
\[
\sup\{|G_{kj, \nu, m}(p)| \mid p\in U_j\cap U_k^* \} \leq A_\nu R^m. 
\]
On $U_j\cap U_k^*$, one has 
\begin{align*}
|H_{kj, n+1}| &\leq  
\left[\!\!\left[\sum_{\nu=2}^n \sum_{m=1}^\infty A_\nu R^m\cdot \left(X + \sum_{\mu=2}^n A_\mu\cdot X^\mu\right)^m\cdot X^\nu \right]\!\!\right]_{n+1} \\
&=
\left[\!\!\left[ \left(\sum_{\nu=2}^n A_\nu X^\nu\right)\cdot \left(\sum_{m=1}^\infty  R^m\cdot \left(X + \sum_{\mu=2}^n A_\mu\cdot X^\mu\right)^m\right) \right]\!\!\right]_{n+1} \\
&\leq  
\left[\!\!\left[ \left(\sum_{\nu=2}^\infty A_\nu X^\nu\right)\cdot \left(\frac{R A(X)}{1-R A(X)}\right) \right]\!\!\right]_{n+1} 
\leq \left[\!\!\left[ \frac{R(A(X))^2}{1-RA(X)}\right]\!\!\right]_{n+1}, 
\end{align*}
and from (\ref{ineq:estim_of_fnu}), one has $I_{0j, n+1}\equiv 0$ when $k=0$, and 
\[
|I_{kj, n+1}| \leq 
\left[\!\!\left[ \sum_{\nu=2}^\infty MR^\nu \cdot (A(X))^\nu \right]\!\!\right]_{n+1} 
=\left[\!\!\left[ \frac{MR^2(A(X))^2}{1-RA(X)} \right]\!\!\right]_{n+1}  
\]
when $k>0$. Therefore, we have 
\[
|J_{kj, n+1}| \leq |H_{kj, n+1}|+|I_{kj, n+1}| \leq \left[\!\!\left[ \frac{R(1+MR)(A(X))^2}{1-RA(X)} \right]\!\!\right]_{n+1}  
\]
on $U_j\cap U_k^*$. 
Since $\{U_0^*, U_1^*, \dots, U_N^*\}$ is an open covering of $C$, 
for each point $p\in U_j\cap U_k$, there exists $\ell\in \{0, 1, 2, \dots, N\}$ such that $p\in U_\ell^*$. 
Then the above estimate shows that 
\[
|J_{kj, n+1}(p)| =|J_{\ell j, n+1}(p) - J_{\ell k, n+1}(p)| 
\leq 2\cdot \max_{a, b}\sup_{U_a\cap U_b^*} |J_{ab, n+1}| \leq \frac{A_{n+1}}{K},
\]
which proves the lemma. 
\end{proof}

\subsection{Proof of Theorem \ref{thm:neighborhood_of_infinite_type}}

For each $j \geq 0$, it follows from Proposition \ref{prop:constr_of_Fjnu} (i), (ii) that 
there exists an open neighborhood $W_j^0 \subset V_j$ of $U_j$ such that the function 

\[
H_j\colon W_j^0 \times \mathbb{C} \to \mathbb{C}, \quad H_j(p, Y) := -w_j(p) + Y + \sum_{\nu=2}^\infty F_{j, \nu}(p)\cdot Y^\nu. 
\]
is holomorphic. 
Moreover for each $p\in U_j$, one has $H_j(p, 0)=0$ and 
\[
\left.\left(\frac{\partial}{\partial Y}H_j(p, Y)\right)\right|_{Y=0} = 1 \not=0. 
\]
Hence the implicit function theorem says that there exists a (Stein) open neighborhood $W_j \subset W_j^0$ of $U_j$ 
and a holomorphic function $u_j : W_j \to \mathbb{C}$ such that $u_j(p)=0$ for each $p\in U_j$ and 
$H_j(p, u_j(p))\equiv 0$, which means that  
\[
w_j = u_j + \sum_{\nu=2}^\infty F_{j, \nu}\cdot u_j^\nu
\]
holds. Proposition \ref{prop:constr_of_Fjnu} (iii) shows that $u_k=u_j$ holds on $W_j \cap W_k$ 
for any $j \geq 1$ and $k \geq 0$, which prove the theorem. 
\qed

\section{Proof of Theorems \ref{thm:main} and \ref{thm:cpt}}\label{section:proof}

\subsection{A conclusion from an argument by changing the model}\label{section:prfpropkey}

In this subsection, we show the following: 
\begin{proposition}\label{prop:key}
Let $C$ and $S$ be as in \S \ref{section:notation}. Assume that $N_{C/S}$ is topologically trivial. 
Then one (and only one) of the following holds: 
\begin{enumerate}
\item[$(a)$] There exists a neighborhood $V$ of $C$ in $S$ and an elliptic fibration $f\colon V \to \Delta$ onto the unit disc $\Delta\subset \mathbb{C}$ such that $C$ is a fiber of $f$. \\
\item[$(b)$] There exists a positive rational number $\ell\in\textstyle\frac{1}{6}\mathbb{Z}$ such that the following assertions $(1)$ and $(2)$ hold: \\
\begin{enumerate}
\item[$(1)$] For each real number $\lambda>1$, there exists a neighborhood $V$ of $C$ and a strongly plurisubharmonic function $\Phi_\lambda\colon V\setminus C\to \mathbb{R}$ such that $\Phi_\lambda(p)$ increases with the same order as $|f_C(p)|^{-\lambda \ell}$ when $p$ approaches $C$ on a neighborhood of each point of $C$, where $f_C$ is a local defining function of $C$. \\
\item[$(2)$] Let $V$ be a neighborhood of $C$ and $\Psi$ be a plurisubharmonic function on $V\setminus C$. If there exists a real number $0<\lambda<1$ such that $\Psi(p)=o(|f_C(p)|^{-\lambda\ell})$ holds on a neighborhood of each point of $C$, then $\Psi|_{V_0\setminus C}$ is a constant function for a neighborhood $V_0$ of $C$ in $V$. 
\end{enumerate}
\end{enumerate}
\end{proposition}

Although the proof of this proposition is essentially given in \cite[p. 875]{K17}, we here explain it again for the convenience of readers (See also the appendix to the present paper). 
The proof is done by applying Ueda's theorems \cite{U83} to a new pair $(\overline{D}, \overline{W})$ of a surface $\overline{W}$ and a non-singular curve $\overline{D}\subset \overline{W}$ obtained from the original pair $(C, S)$ by taking a blowing-up, a finite (ramified) covering, and a blowing-down. 

Denote by $\pi\colon X\to S$ the (minimal) resolution of the ordinary cusp to a normal crossing divisor obtained by blowing-up three times, and by $D$ the pull-back $\pi^*C$. Then one has 
\[
D = 6 C_1 + 3 E_1 + 2E_2 + E_3 
\]
for irreducible divisors $C_1, E_1, E_2$, and $E_3$, where $E_3$ is the strict transform $(\pi^{-1})_*C$ of $C$. 
By \cite[Proposition 7.6, 7.7]{K17}, one can construct a $6: 1$ covering map $p\colon \widetilde{W}\to W$ on a neighborhood $W$ of the support ${\rm supp}\,D$ of $D$ in $X$ and a simple normal crossing divisor
\[
\widetilde{D} = p^{-1}(C_1) + \widetilde{E}_1^{(1)} + \widetilde{E}_1^{(2)} + \widetilde{E}_1^{(3)} + \widetilde{E}_2^{(1)} + \widetilde{E}_2^{(2)} + \widetilde{E}_3^{(1)}
\]
of $\widetilde{W}$ such that $p^*D=6\widetilde{D}$ and that each $\widetilde{E}_\nu^{(\lambda)}$ is a $(-1)$-curve. By Castelnuovo's contraction theorem, we have a morphism $\overline{p}\colon \widetilde{W}\to \overline{W}$ onto a non-singular surface $\overline{W}$ which coincides with the blow-up morphism centered at $6$ points in the (smooth) elliptic curve $\overline{D}:=\overline{p}(p^{-1}(C_1))$. 
Note that, by construction, the normal bundle $N_{\overline{D}/\overline{W}}$ is holomorphically trivial. 
\begin{proof}[Proof of Proposition \ref{prop:key}]
The proof is by cases (based on the type of not $(\widetilde{D}, \widetilde{W})$ but $(\overline{D}, \overline{W})$, see also the appendix to the present paper). 
Denote by $\overline{n}$ the type of the pair $(\overline{D}, \overline{W})$. 

First let us consider the case of $\overline{n} = \infty$. 
As $N_{\overline{D}/\overline{W}}$ is holomorphically trivial, there exists a holomorphic function $h$ on $\overline{W}$, by shrinking $W$, $\widetilde{W}$, and $\overline{W}$ if necessary, such that the zero divisor ${\rm div}(h)$ of $h$ coincides with $\overline{D}$ 
by virtue of \cite[Theorem 3]{U83}. 
As $p^*\pi^*C=p^*D=6\widetilde{D}=6\overline{p}^*\overline{D}$, 
there exists a holomorphic function $g$ on $V' := \pi(W)$ such that ${\rm div}(g)=C$ and 
\[
g\circ \pi\circ p = \prod_{F\in {\rm Aut}(\overline{p})}h \circ \overline{p}\circ F, 
\]
where ${\rm Aut}(\overline{p})$ is the group of deck transformations of $\overline{p}\colon \widetilde{W}\to \overline{W}$. 
As the restriction of $g$ to a neighborhood $V$ of $C$ is proper by Lemma \ref{lem:conti_map_with_cpt_inv}, we obtain a fibration $f$ on $V$ by considering the Stein factorization. 
Since $\overline{D}$ is an elliptic curve, connected components of a general fiber of $h$ are also elliptic curves. 
Therefore $f$ is an elliptic fibration, from which the assertion $(a)$ follows. 

Next let us consider the case of $\overline{n} <\infty$. 
In this case, we show the assertions $(1)$ and $(2)$ in $(b)$ for $\ell := \overline{n}/6$. Take a positive integer $\lambda$. 
When $\lambda>1$, by \cite[Theorem 1]{U83}, there exists a strongly plurisubharmonic function $\vp_{\lambda}$ on $\overline{W}\setminus \overline{D}$ (by shrinking $W$, $\widetilde{W}$, and $\overline{W}$ if necessary) such that $\vp_{\lambda}(p)$ increases with the same order as $|h(p)|^{-\lambda \overline{n}}$ when $p$ approaches $\overline{D}$ on a neighborhood of each point of $\overline{D}$, where $h$ is a local defining function of $\overline{D}$. 
Then 
\[
\widetilde{\vp}_\lambda := \sum_{F\in {\rm Aut}(\overline{p})}\vp_\lambda \circ \overline{p}\circ F 
\]
is an ${\rm Aut}(\overline{p})$-invariant strongly plurisubharmonic function on $\widetilde{W}\setminus \widetilde{D}$ which increases with the same order as $|\widetilde{h}(p)|^{-\lambda \overline{n}}$ when $p$ approaches $\widetilde{D}$ on a neighborhood of each point of $\widetilde{D}$, where $\widetilde{h}$ is a local defining function of $\widetilde{D}$. 
Therefore there exists a function $\Phi_\lambda$ as in the assertion $(1)$ on a neighborhood of $C$, since $p^*\pi^*C=p^*D=6\widetilde{D}$. 
When $\lambda<1$, take a function $\Psi$ as in the assertion $(2)$. 
Then, by the same argument as above, we have a plurisubharmonic function $\psi$ on $\overline{p}((\pi\circ p)^{-1}(V))\setminus \overline{D}$ such that $\Psi\circ \pi\circ p = \vp \circ\overline{p}$. 
Again by $p^*\pi^*C=p^*D=6\widetilde{D}$ and $\widetilde{D}=\overline{p}^*\overline{D}$, we have $\vp(p)=o(d(p, \overline{D})^{-\lambda \overline{n}})$ as $p\to \overline{D}$, where $d(p, \overline{D})$ is the (local) Euclidean distance from $p$ to $\overline{D}$. Thus the assertion $(2)$ follows from \cite[Theorem 2]{U83}. 

Note that the assertions $(a)$ and $(b)$ cannot hold simultaneously, since any plurisubharmonic function on $V\setminus C$ is $f$-fiberwise constant by the maximum principle if $(a)$ holds. 
\end{proof}

\subsection{Proof of Theorem \ref{thm:main}}
When $N_{C/S}$ is not holomorphically trivial, 
it follows from (the contraposition of) Lemma \ref{lem:fibr_case_normal_bdl_triv} that Proposition \ref{prop:key} $(a)$ never holds. Thus Proposition \ref{prop:key} $(b)$ holds in this case. 
In what follows, we show the theorem by assuming that $N_{C/S}$ is holomorphically trivial. Denote by $n$ the type of the pair $(C, S)$. 
When $n=\infty$, it follows from Theorem \ref{thm:neighborhood_of_infinite_type} that Proposition \ref{prop:key} $(a)$ holds. 
When $n<\infty$, it follows from Lemma \ref{lem:well-def_type} that 
Proposition \ref{prop:key} $(a)$ never holds, from which we have the assertion $(b)$. 
\qed

\subsection{Proof of Theorem \ref{thm:cpt}}
First let us consider the case where $N_{C/S}$ is holomorphically trivial and the pair $(C, S)$ is of infinite type. 
In this case, by virtue of Theorem \ref{thm:main} $(i)$, there exists an elliptic fibration $P\colon V\to \Delta$ on a neighborhood $V$ of $C$ in $S$. 
Take a general fiber $Y=P^{-1}(q)$ of $P$. 
Then $Y$ is a non-singular compact connected hypersurface of $S$ whose normal bundle is holomorphically trivial. 
Let $r$ be a sufficiently small positive number and $\chi\colon \mathbb{R}\to \mathbb{R}$ be a non-decreasing convex function of class $C^\infty$ such that $\chi(t)=t$ for $t\in (-\log r, \infty)$ and that $\chi\equiv c$ holds for a constant $c\in\mathbb{R}$ on $(-\infty, -\textstyle\frac{1}{2}\log r)$. 
Denote by $\psi$ the function on $S\setminus Y$ defined by 
\[
\psi(x) := \begin{cases}
\chi(-\log|P(x)-q|) & \text{if}\ x \in V\\
c & \text{if}\ x\not\in V
\end{cases}. 
\]
Then, by considering the function $x\mapsto \log(1+e^{\psi(x)})$, it follows from \cite[Proposition 2.1]{K22} that the line bundle $[Y]$ on $S$ is semi-positive: i.e. $[Y]$ admits a $C^\infty$ Hermitian metric whose Chern curvature is positive semi-definite at any point of $S$. Thus, 
by \cite[Theorem 1.1 $(i)$]{K22}, 
there exists a fibration $f\colon S \to R$ onto a compact Riemann surface $R$ such that $Y$ is a fiber of $f$. 
As the image $f(C)$ is a connected analytic subset of $R$, it follows that either $f(C)$ is a point or $f(C)=R$. As $C\cap Y=\emptyset$ and $Y$ is a fiber of $f$, the image $f(C)$ never includes the image of $Y$, from which the assertion $(i)$ follows. 
As $S\setminus C$ is the total space of an elliptic fibration over an open Riemann surface, we have that $S\setminus C$ is holomorphically convex in this case. 
In the other case, it follows from the assertion $(1)$ of 
Theorem \ref{thm:main} $(ii)$ that there exist a neighborhood $V$ of $C$ and a strongly plurisubharmonic function $\Phi\colon V\setminus C\to \mathbb{R}$ such that $\Phi(p)\to \infty$ as $p\to C$. 
Fix a relatively compact neighborhood $V_0$ of $C$ in $V$. 
Set $c:=\max\{1, \max_{\partial V_0}\Phi\}$ and denote by $\varphi$ the function on $S\setminus C$ defined by 
\[
\varphi(x) := \begin{cases}
\max\{2c, \Phi(x)\} & \text{if}\ x\in V_0\setminus C\\
2c & \text{if}\ x\in S\setminus V_0
\end{cases}. 
\]
Then, by construction, we have that $\varphi$ is an exhaustion plurisubharmonic function on $S\setminus C$ which is strongly plurisubharmonic on a neighborhood of the boundary of $S\setminus C$. 
Thus $S\setminus C$ is strongly $1$-convex, and it follows from Grauert's theorem \cite{G} that $S\setminus C$ is a proper modification of a Stein space, from which it follows that it is holomorphically convex. 

Finally let us show the equivalence of the assertions $(1), (2)$, and $(3)$. 
The implication $(2)\implies (1)$ has been already proven. 
To show the implication $(1)\implies (3)$, assume the assertion $(i)$. 
Then $[C]=p^*L$ holds for a line bundle $L$ on $R$ of positive degree. 
As $L$ has a Hermitian metric $h_L$ whose Chern curvature is positive, the assertion $(3)$ follows by considering $p^*h_L$. 
Finally we show the implication $(3)\implies (2)$. 
We show this by contradiction. 
Assume that $(2)$ does not hold and $(3)$ holds. 
Then, by using a Hermitian metric $h$ on $[C]$ with $\sqrt{-1}\Theta_h\geq 0$, one can construct a plurisubharmonic function $\vp\colon S\setminus C\to \mathbb{R}$ by letting
\[
\vp(x) := -\log |f_C(x)|_h^2, 
\]
where $f_C$ is the canonical section of $[C]$. 
As $\vp$ increases with the logarithmic order along $C$, 
a contradiction follows from the assertion $(2)$ of Theorem \ref{thm:main} $(ii)$. 
\qed

\section{Examples and further discussions}\label{section:examples}

\subsection{A rational curve with an ordinary cusp in the blow-up of $\mathbb{P}^2$ at $9$ points}

In this subsection, we let $S$ be a rational surface obtained by blowing-up $\mathbb{P}^2$ at $9$ points such that the linear system of its anti-canonical divisor includes a rational curve $C$ with an ordinary cusp. 
Note that $N_{C/S}$ is topologically trivial. 
For this example, we have a sufficient and necessary condition for the assertion $(i)$ in Theorem \ref{thm:cpt} to hold as follows: 
\begin{lemma}\label{lem:ex_P2bl9pts}
Let $S$ be a rational surface obtained by blowing-up $\mathbb{P}^2$ at $9$ points such that the linear system of its anti-canonical divisor includes a rational curve $C$ with an ordinary cusp. 
Then the assertion $(i)$ in Theorem \ref{thm:cpt} holds if and only if $N_{C/S}$ is holomorphically trivial. 
\end{lemma}

\begin{proof}
By Lemma \ref{lem:fibr_case_normal_bdl_triv}, $N_{C/S}$ is holomorphically trivial if the assertion $(i)$ in Theorem \ref{thm:cpt} holds. 
Assume $N_{C/S}$ is holomorphically trivial. Then, from the short exact sequence obtained by tensoring $\mathcal{O}_S(C)$ with the defining short exact sequence of the defining ideal sheaf $I_C (\cong \mathcal{O}_S(-C))$ of $C$, the following exact sequence is induced: 
\begin{equation}\label{eq:exseq_in_prf_lem:ex_P2bl9pts}
H^0(S, \mathcal{O}_S(C)) \to H^0(C, \mathcal{O}_S(C)|_C) \to H^1(S, \mathcal{O}_S). 
\end{equation}
Since it follows from the assumption that $\mathcal{O}_S(C)|_C=\mathcal{O}_C(N_{C/S})\cong \mathcal{O}_C$, there exists a nowhere vanishing section $g\in H^0(C, \mathcal{O}_S(C)|_C)$. 
As $S$ is rational, one has $H^1(S, \mathcal{O}_S)=0$. 
Thus, by the exactness of (\ref{eq:exseq_in_prf_lem:ex_P2bl9pts}), one can find a section $G\in H^0(S, \mathcal{O}_S(C))$ such that $G|_C = g$. In particular, the zeros of $G$ never intersect $C$. 
The fibration $f\colon S\to R$ as in 
Theorem \ref{thm:cpt} $(i)$ can be constructed from the meromorphic function 
\[
S \ni x  \mapsto \frac{\sigma_C(x)}{G(x)}, 
\]
where $\sigma_C\in H^0(S, \mathcal{O}_S(C))$ is the canonical section. 
\end{proof}

Motivated by the results and examples due to Kazama and Takayama \cite{KT1} \cite{KT2}, 
$\partial\overline{\partial}$-problem (whether $\partial\overline{\partial}$-lemma holds) on $S\setminus C$ was investigated in \cite{K23} for the blow-up $S$ of $\mathbb{P}^2$ at $9$ points and an anti-canonical divisor $C$ when $C$ is non-singular. 
Related to this example, we believe it is worth asking the following question for $(C, S)$ as in Lemma \ref{lem:ex_P2bl9pts}:
\begin{question}
Let $S$ and $C$ be as in Lemma \ref{lem:ex_P2bl9pts}. 
Does $\partial\overline{\partial}$-lemma hold (in the sense of \cite{KT1}) on $S\setminus C$? 
\end{question}

\subsection{Further discussions}\label{subsection_futherdiscuss}

First let us pose the following: 
\begin{question}
Does there exist an example $(C, S)$ of a rational curve $C$ with an ordinary cusp in a non-singular complex surface $S$ such that $N_{C/S}$ is holomorphically trivial and the pair $(C, S)$ is of finite type? 
\end{question}

It seems worth asking this question also under the assumption that $S$ is compact (and K\"ahler, or projective), since at least the rational surface $S$ as in the previous subsection never admits a rational curve $C$ with an ordinary cusp such that $N_{C/S}$ is holomorphically trivial and the pair $(C, S)$ is of finite type by virtue of Lemma \ref{lem:ex_P2bl9pts}. 

Next, using the notation in \S \ref{section:prfpropkey}, 
let us compare the type $n$ of the pair $(C, S)$, 
the type $\widetilde{n}$ of $(\widetilde{D}, \widetilde{W})$ (in the sense of \cite{K17}), 
and $\overline{n}$ of $(\overline{D}, \overline{W})$. 

Take an open covering $\{V_j\}$ of a neighborhood of $C$ in $S$ as in \S \ref{section:notation} and a defining function $w_j\colon V_j\to \mathbb{C}$ of $C\cap V_j$ for each $j$ such that $\{(V_j, w_j)\}_j$ is of type $n$. 
Then, by considering the refinement of the open covering $\{(\pi\circ p)^{-1}(V_j)\}_j$ and the restrictions of brunches of $(w_j\circ \pi\circ p)^{1/6}$, one can construct a system of local defining functions of $\widetilde{D}$ of type $n$ (in the sense of \cite{K17}). 
Therefore 
\begin{equation*}
n \leq \widetilde{n}
\end{equation*}
holds. 
By applying the same argument to a system of local defining functions of $\overline{D}$ of type $\overline{n}$ and the morphism $\overline{p}\colon \widetilde{W}\to \overline{W}$, we have 
\begin{equation*}
\overline{n} \leq \widetilde{n}. 
\end{equation*}
Based on the above discussion, let us pose the following: 
\begin{question}\label{question:comparison_ueda_types}
How is the relation between $n$ and $\overline{n}$? 
\end{question}

Finally, let us discuss the consequence of Theorem \ref{thm:main} from the viewpoint of Ueda's classification \cite[\S 5]{U83} of neighborhoods of compact non-singular curves whose normal bundles are topologically trivial. 
Let $S$ and $C$ be as in Theorem \ref{thm:main}. 
We will restrict ourselves to the case where $N_{C/S}$ is holomorphically trivial in what follows 
(Recall that any topologically trivial line bundle on a compact non-singular curve is unitary flat (a theorem of Kashiwara, see \cite[\S 1]{U83} for example). 
We refer to \cite[Theorem 1, 2]{U91} and \cite[Theorem 1.6]{K17} for the case where a curve has nodes and the normal bundle is topologically trivial and holomorphically non-trivial). 

The case where the assertions $(i)$ and $(ii)$ of Theorem \ref{thm:main} hold can be naturally interpreted as an analogue of the classes ($\beta'$) and $(\alpha)$ in Ueda's classification, respectively. 
From this viewpoint, The assertion $(i)$ of Theorem \ref{thm:main} (and Theorem \ref{thm:neighborhood_of_infinite_type}) can be regarded as an analogue of \cite[Theorem 3]{U83}, and $(ii)$ can be regarded as an analogue of \cite[Theorem 1, 2]{U83}. 
As any unitary flat line bundle on $C$ is holomorphically trivial by Lemma \ref{lem:Pic0}, it can be said that there is no case which corresponds to the classes ($\beta''$) and ($\gamma$) in Ueda's classification for $(C, S)$.  

\begin{question}
How can we generalize Theorem \ref{thm:main} for a compact curve with (general) cusps embedded in a non-singular surface? Can one find an example which can be regarded as an analogue of the class ($\beta''$) or ($\gamma$) in Ueda's classification when $C$ is a compact curve with cusps? 
\end{question}





\vskip8mm

\appendix
\section{Erratum to ``Ueda theory for compact curves with nodes"}
This appendix aims to correct some errors in some statements and arguments in \cite{K17} which are mainly needed to show \cite[Theorem 7.1]{K17}. 
The statements of all the theorems in this paper, including Theorem 7.1, remain unchanged.

In the proof of Lemma 7.6, $(L')^a\otimes\mathcal{O}_W(\textstyle\sum_{\nu=1}^Na_\nu E_\nu)$ should read $(L')^{-a}\otimes\mathcal{O}_W(\textstyle\sum_{\nu=1}^Na_\nu E_\nu)$. 
In the statements of Propositions 7.8 and 7.9, $(\widetilde{W}, \widetilde{D})$ should read $(\overline{W}, \overline{D})$. 
The fourth and fifth sentences in the proof of Proposition 7.8 should be disregarded. 
In the assertion $(i)$ of Proposition 7.9, ``$\Phi_\lambda(p) = O(d(p, {\rm supp}\,D)^{-\lambda n/a})$ hold as $p\to {\rm supp}\,D$, where $d(p, {\rm supp}\,D)$ is the distance from $p$ to ${\rm supp}\,D$ calculated by using a local Euclidean metric on a neighborhood of a point of $C$ in $V$" should read ``$\Phi_\lambda(p) = O(|f_D(p)|^{-\lambda n/a})$ holds on a neighborhood of any point of $C$, where $f_D$ is a local defining function of $D$". 
Similarly, ``$\Psi(p)=O(d(p, {\rm supp}\,D)^{-\lambda n/a})$ as $p\to {\rm supp}\,D$" in Proposition 7.9 $(ii)$ should read ``$\Psi(p)=O(|f_D(p)|^{-\lambda n/a})$ holds on a neighborhood of any point of $C$". 
In p. 875 l. 22 and 27, $(\widetilde{W}, \widetilde{D})$ should read $(\overline{W}, \overline{D})$. 


\begin{thebibliography}{99}
\bibitem[GS]{GS} \textsc{X. Gong, L. Stolovitch}, Equivalence of Neighborhoods of Embedded Compact Complex Manifolds and Higher Codimension Foliations, Arnold Math J. {\bf 8} (2022), 61--145. 
\bibitem[G]{G} \textsc{H. Grauert}, \"Uber Modifikationen und exzeptionelle analytische Mengen, Math. Ann., {\bf146} (1962), 331--368. 
\bibitem[HM]{HM} \textsc{J. Harris, I. Morrison}, Moduli of curves, Graduate Texts in Mathematics {\bf 187}, Springer-Verlag, 1998.
\bibitem[KT1]{KT1} \textsc{H. Kazama, S. Takayama}, $\partial\overline{\partial}$-problem on weakly 1-complete K\"ahler manifolds, Nagoya Math. J. {\bf 155} (1999), 81--94. 
\bibitem[KT2]{KT2} \textsc{H. Kazama, S. Takayama}, On the $\partial\overline{\partial}$-equation over pseudoconvex K\"ahler manifolds, Manuscripta Math. {\bf 102} (2000), 25--39.
\bibitem[KS]{KS} \textsc{K. Kodaira, D. C. Spencer}, A theorem of completeness o f characteristic systems of complete continuous systems, Amer. J. Math., {\bf 81} (1959), 477-500. 
\bibitem[K1]{K17} \textsc{T. Koike}, Ueda theory for compact curves with nodes, Indiana U. Math. J, {\bf 66}, 3 (2017), 845--876. 
\bibitem[K2]{K20} \textsc{T. Koike}, Higher codimensional Ueda theory for a compact submanifold with unitary flat normal bundle, Nagoya Math. J., {\bf 238} (2020), 104--136. 
\bibitem[K3]{K22} \textsc{T. Koike}, On the complement of a hypersurface with flat normal bundle which corresponds to a semipositive line bundle, Math. Ann., {\bf 383} (2022), 291--313. 
\bibitem[K4]{K23} \textsc{T. Koike}, $\overline{\partial}$ cohomology of the complement of a semi-positive anticanonical divisor of a compact surface, arXiv:2308.03761. 
\bibitem[N]{N} \textsc{A. Neeman}, Ueda theory: theorems and problems, Mem. Amer. Math. Soc. {\bf 81}, no. 415 (1989).
 \bibitem[S]{S} \textsc{Y. T. Siu}, Every Stein subvariety admits a Stein neighborhood, Invent. Math., {\bf 38} (1976), 89--100.
 \bibitem[U1]{U83} \textsc{T. Ueda},  On the neighborhood of a compact complex curve with topologically trivial normal bundle, J. Math. Kyoto Univ., {\bf 22} (1983), 583--607. 
 \bibitem[U2]{U91} \textsc{T. Ueda}, Neighborhood of a rational curve with a node, Publ. RIMS, Kyoto Univ., {\bf 27} (1991), 681--693. 
  \bibitem[U3]{U} \textsc{T. Ueda}, On the neighborhood of a compact complex curve with singularities, to appear in RIMS kokyuroku (in Japanese).
\end{thebibliography}
\end{document}